\documentclass[finalversion]{FPSAC2020}
\usepackage{lipsum}
\usepackage{amsmath}

\usepackage[utf8]{inputenc}
\usepackage[T1]{fontenc}
\usepackage{amssymb}
\usepackage{calc}
\usepackage{ragged2e}
\usepackage{xfrac}

\usepackage{amsthm}
\usepackage{amsfonts}
\usepackage{amssymb}
\usepackage{tikz}
\usepackage{color}
\usepackage{ytableau}
\usepackage{multicol}
\usepackage{mathtools}

\newtheorem{teo}{Theorem}[section]

\newtheorem{prop}[teo]{Proposition}
\newtheorem{lema}[teo]{Lemma}

\theoremstyle{definition}

\newtheorem{defin}[teo]{Definition}

\theoremstyle{remark}
\newtheorem{ex}[teo]{Example}

\DeclareMathSymbol{\shortminus}{\mathbin}{AMSa}{"39}

\title[An action of the cactus group on shifted tableau crystals]{An action of the cactus group on shifted tableau crystals}

\author[Inês Rodrigues]{Inês Rodrigues\thanks{\href{mailto:imarrodrigues@fc.ul.pt}{imarrodrigues@fc.ul.pt}. The author is partially supported by the Lisbon Mathematics PhD program (funded by the Portuguese Science Foundation). This research was made within the activities of the Group for Linear, Algebraic and Combinatorial Structures of the Center for Functional Analysis, Linear Structures and Applications (University of Lisbon), and was partially supported by the Portuguese Science Foundation through the Strategic Project UID/MAT/04721/2013.}\addressmark{1}}

\address{\addressmark{1} Center for Functional Analysis, Linear
Structures and Applications, Faculty of Sciences, University of Lisbon, Portugal}

\received{\today}

\abstract{Recently, Gillespie, Levinson and Purbhoo introduced a crystal-like structure for shifted tableaux, called the shifted tableau crystal. We introduce a shifted analogue of the crystal reflection operators, which coincides with the restriction of the shifted Schützenberger involution to any primed interval of two adjacent letters. Unlike type $A$ Young tableau crystals, these operators do not realize an action of the symmetric group on the shifted tableau crystal because  braid relations do not hold. We exhibit a natural internal action of the cactus group, realized by restrictions of the shifted Schützenberger involution on primed intervals of the underlying crystal alphabet.}

\resume{Gillespie, Levinson et Purbhoo ont récemment introduit une structure similaire à celle de graphe cristalline sur les tableaux déplacés,  appelée graphe cristalline des tableaux déplacés. Nous introduisons une analogue de l'opérateur de réflexion de graphe cristalline, qui coïncide avec la restriction de l'involution de Schützenberger déplacée aux  intervalles marqués de deux lettres adjacentes. Contrairement aux graphes cristallines  de type $A$, des tableaux de Young, ces opérateurs ne réalisent pas une action du groupe symétrique sur le  graphe cristalline des tableaux déplacés parce que les rélations de tresse ne sont pas satisfaites. Pourtant, nous montrons qu'il y a une action interne naturelle du groupe de cactus, réalisée par les restrictions de l'involution de Schützenberger déplacé aux intervalles marqués de l'alphabet sous-jacent au graphe cristalline.}

\keywords{Shifted tableaux, Schützenberger involution, crystal bases, cactus group}

\begin{document}
\maketitle
%% note that you DO NOT have to put your abstract here -- it is generated by \maketitle and the \abstract and \resume commands above

\section{Introduction}

Young tableaux, as well as shifted tableaux, arise in many areas of mathematics. While the first have their original role in the representation theory of symmetric groups, the latter have their origin in projective representations, due to I. Schur. One important tool for the study of the former are Kashiwara crystals \cite{Kash95}. We recall that a \emph{Kashiwara crystal} of type $A$ (for $GL_n$) is a non-empty set $\mathcal{B}$ together with maps $e_i, f_i : \mathcal{B} \longrightarrow \mathcal{B} \sqcup \{\emptyset\}$, length functions $\varepsilon_i, \varphi_i : \mathcal{B} \longrightarrow \mathbb{Z}$, for $i \in I= [n-1]$, and weight function $wt: \mathcal{B} \longrightarrow \mathbb{Z}^n$ satisfying certain axioms (see \cite[Definition 2.13]{BumpSchi17}). This crystal may be regarded as a directed graph, with vertices  in $\mathcal{B}$ and $i$-coloured edges $y \xrightarrow{i} x$ if and only if $f_i (y) = x$, for $i \in I$. The set of semistandard Young tableaux of a given shape in the alphabet $[n]$ is known to provide a model for Kashiwara type $A$ crystals \cite[Chapter 3]{BumpSchi17}, with coplactic operators $e_i$ and $f_i$ defined in terms of reading words. This crystal is isomorphic to the crystal basis of an irreducible $U_q(\mathfrak{gl}_n)$-module. The Schützenberger involution \cite{Schu76} (also known as Lusztig involution) is defined on the type $A$ Young tableau crystal \cite{BumpSchi17}, and acts on its graph structure by "flipping" it upside down, while reverting the orientation of arrows and their colours. It is realized by the evacuation (for straight shapes) or its coplactic extension, often called reversal (for skew shapes).

Recently, Gillespie, Levinson and Purbhoo \cite{GL19, GLP17} introduced a crystal-like structure on shifted tableaux, called the \emph{shifted tableau crystal}. The vertices of this structure are the skew shifted tableaux, for a given shape $\lambda/\mu$ on the primed alphabet $ [n]' $ and it has double edges, corresponding to the action of the primed and unprimed lowering and raising operators which commute with the shifted \textit{jeu de taquin}. Each connected component has a unique highest weight element (a vertex for which all the raising operators are equal to $\emptyset$), a Littlewood-Richardson-Stembridge (LRS) tableau of shape $\lambda/\mu$, and a unique lowest weight element, its reversal. We remark that this structure is not a queer crystal, differing from the one in \cite{AsOg18,GHPS18}, which is indeed a crystal for the queer Lie superalgebra.

We introduce a shifted analogue of the crystal reflection operator in type $A$, for each $i \in I$, for the shifted tableau crystal (Definition \ref{shrefop}). It coincides with the restriction of the shifted Schützenberger involution on the letters $\{i',i,(i+1)',i+1\}$ (see Figure \ref{fig:flip}). Unlike type $A$ crystals, they do not define an action of the symmetric group on the shifted tableau crystal, since the braid relations do not need to be satisfied. We then show that the restriction of the shifted Schützenberger involution  to  all primed subintervals of $[n]$ yields an action of the cactus group on that crystal. This paper has the following structure: Section \ref{sec2} provides the basic notions on shifted tableaux. Then, Section \ref{sec3} gives the main concepts on the shifted tableau crystal of \cite{GL19,GLP17}, and introduces the shifted crystal reflection operators. In Section \ref{sec4}, we prove the main result (Theorem \ref{cactusaction}), where a natural action of the cactus group in the shifted tableau crystal, realized by those restrictions of the shifted Schützenberger involution, is exhibited (see Figure \ref{fig:crystal}).

\section{Background}\label{sec2}

This section is intended to provide the basic definitions and results on shifted tableaux, words, and involutions among them. We follow the notations in \cite{GL19, GLP17}. A \emph{strict partition} is a sequence $\lambda = (\lambda_1, \ldots, \lambda_{\ell(\lambda)})$ of non-negative integers such that $\lambda_1 > \ldots > \lambda_{\ell(\lambda)}$. It may be represented, in English notation, by a \emph{shifted shape} $S(\lambda)$ which consists of $|\lambda|$ boxes placed in $\ell(\lambda)$ rows, with the $i$-th row having $\lambda_i$ boxes and being shifted $i-1$ units to the right. Skew shapes $S(\lambda/\mu)$ are defined as usual. Shapes of the form $\lambda/\emptyset$ are called \emph{straight}. A shifted shape $\lambda$ lies naturally in a \emph{stair shifted shape} $\delta = (\lambda_1, \lambda_1-1, \ldots, 1)$. The complement of the shape of $\lambda$ in $\delta$ defines the shifted partition $\lambda^{\vee}$, the \emph{complement} of $\lambda$ (see Figure \ref{fig:1}).
\begin{figure}
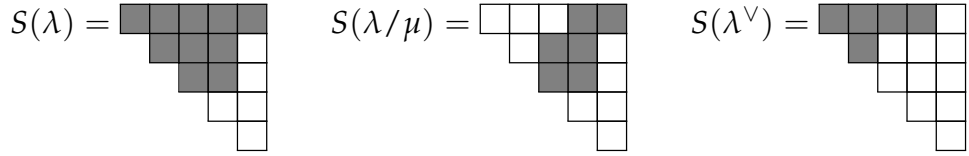

\begin{center}
$S(\lambda)=
	\ytableausetup{smalltableaux}	
	\begin{ytableau}
    *(gray)  & *(gray)  &*(gray)  &*(gray)  &*(gray)  \\
    \none & *(gray)  & *(gray)  & *(gray) &   \\
    \none & \none & *(gray) & *(gray)  & \\
    \none & \none & \none & & \\
    \none & \none & \none & \none &
  \end{ytableau}
  \qquad
  S(\lambda/\mu)=
	\ytableausetup{smalltableaux}	
	\begin{ytableau}
    {}  & {} & {} &*(gray)  &*(gray)  \\
    \none & {}  & *(gray)  & *(gray) &   \\
    \none & \none & *(gray) & *(gray) & \\
    \none & \none & \none & & \\
    \none & \none & \none & \none &
  \end{ytableau}
  \qquad
 S(\lambda^{\vee})=
	\ytableausetup{smalltableaux}	
	\begin{ytableau}
    *(gray)  & *(gray)  &*(gray)  &*(gray)  &   \\
    \none & *(gray)  &    &  &   \\
    \none & \none &   &  & \\
    \none & \none & \none & & \\
    \none & \none & \none & \none &
  \end{ytableau}
  $
 \end{center}
 \caption{The shapes of $\lambda$, $\lambda/\mu$ and $\lambda^{\vee}$, shaded in gray, for $\lambda= (5,3,2)$ and $\mu = (3,1)$.}
 \label{fig:1}
 \end{figure}
Consider the totally ordered alphabet $[n] = \{1 < \ldots < n\}$ and define the \textit{primed} (or marked) alphabet $[n]'$ as $\{1' \!<\! 1 \! <\! \ldots \! < \! n'\! < \! n\}$. We write $\mathbf{i}$ when referring to the letters $i$ and $i'$ without specifying whether they are primed. The \emph{canonical form} of a string $w$ in $[n]'$ is the string obtained from $w$ by replacing the leftmost $\mathbf{i}$ (if it exists) with $i$, for all $1 \leq i \leq n$. Strings $w$ and $v$ are said to be \emph{equivalent}, denoted by $w \simeq v$, if they have the same canonical form. We have $12'2'1123'2'2 \simeq 122'11232'2$, the latter being the canonical form of the former. A \emph{word} $\hat{w}$ is an equivalence class of strings \cite[Definition 2.2]{GLP17}. The \emph{canonical representative} is given by the canonical form of all the elements of $\hat{w}$. The \emph{weight} of a word $\hat{w}$ is $wt(\hat{w}) = (wt_1, \ldots, wt_n)$, where $wt_i$ is equal to the number of $i$ and $i'$ in $\hat{w}$. We denote $wt(\hat{w})^{\mathsf{rev}} = (wt_n, \ldots, wt_1)$.

Let $\lambda$ and $\mu$ be strict partitions such that $\mu \subseteq \lambda$. A \emph{semistandard shifted (Young) tableau} $T$ of shape $\lambda / \mu$ is a filling of $S(\lambda/\mu)$ with letters in $\{1' \!<\! 1 \!<\! \ldots\}$ such that the entries are weakly increasing in each row and in each column and there is at most one $i'$ per row and one $i$ per column. The \emph{(row) reading word} $w(T)$ of such a tableau is formed by reading the entries of $T$ from left to right, from bottom to top. The \emph{weight} of $T$ is defined as $wt(T):=wt(w(T))$. A shifted tableau is said to be \emph{standard} if its weight is $(1, \ldots, 1)$. We say that a tableau $T$ is in \emph{canonical form} if $w(T)$ is the canonical representative of $\widehat{w(T)}$. A tableau $T$ in canonical form is identified with its set of \emph{representatives}, that are obtained by possibly priming the entry corresponding to the leftmost $i$ in $w(T)$, for all $i$. Let $\mathsf{SShT}(\lambda/\mu, n)$ denote the set of semistandard shifted tableaux of shape $\lambda/\mu$ in the alphabet $[n]'$ (in canonical form). The following is a semistandard shifted tableau:
\begin{center}
$T=\ytableausetup{smalltableaux}
\begin{ytableau}
{} & {} & {} & 1 &1 & 2'\\
\none & {} & 2 & 3' & 3\\
\none & \none & 3 & 3
\end{ytableau}$, where
$w(T)=3323'3112' $ and
$wt(T)=(2,2,4)$.
\end{center}

\subsection{The shifted \textit{jeu de taquin}, Knuth equivalence and dual equivalence}\label{sub22}

The shifted \emph{jeu de taquin} \cite{Sag87,Wor84}, for shifted tableaux, is similar to the one for usual Young tableaux. Given  $T\in \mathsf{SShT}(\lambda/\mu, n)$, an \emph{inner jeu de taquin slide} is a process in which an inner corner is chosen and then either the entry to its right or the one below it is chosen to slide, in such way that the tableau is still semistandard, and then repeating the process with the obtained empty square until it is an outer corner. An \emph{outer jeu de taquin slide} is the reverse process, starting with an outer corner. This process has an exception, illustrated by the following slide:

\begin{center}
\includegraphics[scale=0.4]{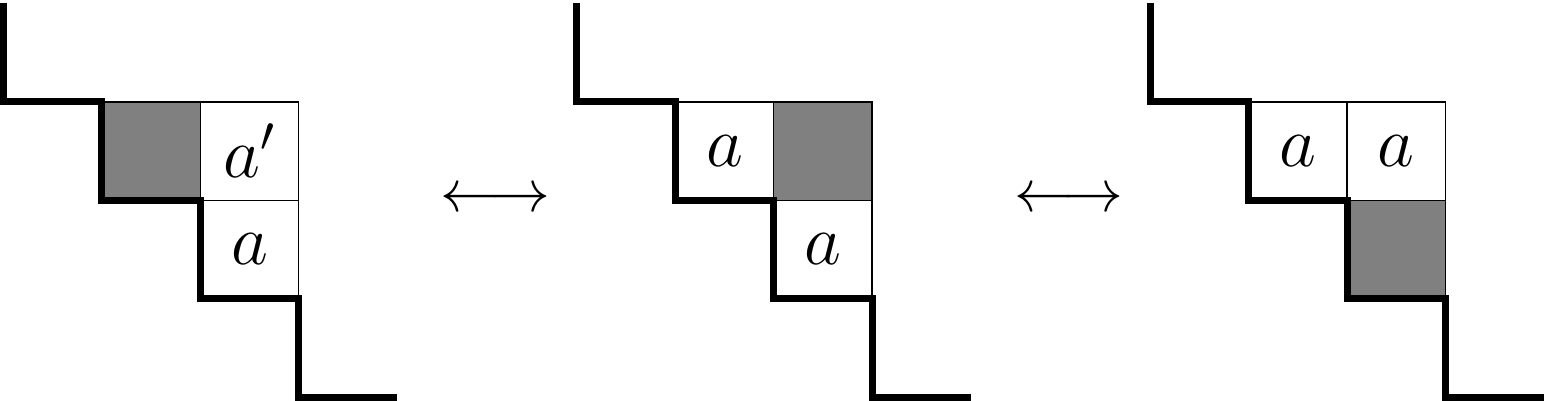}
\end{center}

The \emph{rectification} $rect(T)$ of $T$ is the tableau obtained by applying a sequence of inner slides until a straight shape is obtained (this does not depend on the chosen sequence of slides \cite[Theorem 11.1]{Sag87}). The rectification of a word $w$ is the word of the rectification of any tableau with reading word $w$. Two tableaux are said to be \emph{shifted jeu de taquin equivalent} if they have the same rectification. An operator on shifted tableaux is called \emph{coplactic} if it commutes with the shifted \textit{jeu de taquin}.

The \emph{standardization} of a word $w$, denoted $\mathsf{std}(w)$, is obtained by replacing the letters of any representative of $w$ with $1, \ldots, n$, from least to greatest, reading right to left for primed entries, and left to right for unprimed entries \cite[Definition 2.8]{GLP17}. This process does not depend on the choice of the representative. The standardization of a shifted tableau $T$ is the tableau of the same shape as $T$ with reading word $\mathsf{std}(w(T))$. 

Given $\nu$ a strict partition, the \emph{Yamanouchi tableau} $Y_{\nu}$ is the tableau of shape $\nu$ whose $i$-th row is filled only with $i$'s. It is easy to check that $Y_{\nu}$ is the unique tableau, up to canonical form, that has shape and weight equal to $\nu$. A \emph{Littlewood-Richardson-Stembridge (LRS) tableau} is a tableau $T \in \mathsf{SShT}(\lambda/\mu,n)$ such that $rect(T)=Y_{\nu}$ \cite{Wor84}. Its reading word is called a \emph{ballot} (or \emph{lattice}) word. Given strict partitions $\lambda$, $\mu$ and $\nu$, such that $|\lambda| = |\mu| + |\nu|$, the \emph{shifted Littlewood-Richardson} coefficients $f_{\mu\nu}^{\lambda}$ are defined as the number of LRS tableaux of shape $\lambda/\mu$ and weight $\nu$. For other formulations, see \cite{CNO14,Stem89}.

\begin{defin}
Two words $w$ and $v$ are said to be \emph{shifted Knuth equivalent}, denoted $w \equiv_k v$, if one can be obtained from the other by applying a sequence of the following Knuth moves on adjacent letters:

\textbf{(K1)} $bac \longleftrightarrow bca$ if, under the standardization ordering, $a < b < c$.

\textbf{(K2)} $acb \longleftrightarrow cab$ if, under the standardization ordering, $a < b < c$.

\textbf{(S1)} $ab \longleftrightarrow ba$ if these are the first two letters.

\textbf{(S2)} $aa \longleftrightarrow aa'$ if these are the first two letters.
\end{defin}

Two semistandard shifted tableaux are shifted Knuth equivalent if their reading words are shifted Knuth equivalent \cite[Theorem 12.2]{Sag87}\label{jdtknuth}, or, equivalentely, if they have the same rectification \cite[Theorem 6.4.17]{Wor84}.  Two semistandard shifted tableaux are \emph{shifted dual equivalent} if they have the same shape after applying any sequence of shifted \textit{jeu de taquin} slides to both.

\subsection{The shifted evacuation and reversal}\label{subsectevacrev}\label{sub23}

In this section, we recall an involution on semistandard shifted tableaux of straight shape, known as the \emph{(shifted) evacuation}, which preserves shape and reverts weight. This involution was firstly presented by Worley \cite{Wor84}, as an analogue of the Schützenberger involution \cite{Schu76}. We use Worley's definition, however, we remark that Choi, Nam and Oh \cite[Section 5]{CNO17} gave another formulation using the shifted switching process, and proved the coincidence of both.

Given $T\in \mathsf{SShT}(\lambda/\mu, n)$, the tableau $T^*$ is obtained by reflecting $T$ along the anti-diagonal in the shifted stair shape $\delta = (\lambda_1, \lambda_1 -1, \ldots, 1)$, while complementing the entries by $i \mapsto (n-i+1)'$ and $i' \mapsto n-i+1$. Note that if $T$ is of shape $\lambda/\mu$, then $T^*$ is of shape $ \mu^{\vee} / \lambda^{\vee}$, and $wt(T^*) = wt(T)^{\mathsf{rev}}$. If $T$ is a tableau with straight shape, then $T^E := rect (T^*)$ and the operator $E$ is called the \emph{(shifted) evacuation} \cite[ Definition 7.1.5]{Wor84}. For example, if $T= \begin{ytableau}
1 & 1 & {2'} & 2\\
\none & 2 & 3
\end{ytableau}$, then 
$T^E = \begin{ytableau}
1 & {2'} & 2 & 3\\
\none & 2 & 3
\end{ytableau}$. Moreover, $T^E$ has the same shape as $T$ and $(T^E)^E = T$ \cite[Lemma 7.1.6]{Wor84}. As a consequence of the uniqueness of $Y_{\nu}$ we have that $Y_{\nu}^E$ is the unique one of shape $\nu$ and weight $\nu^{\mathsf{rev}}$.

It is due to Haiman that, given $T \in \mathsf{SShT}(\lambda/\mu,n)$, there is a unique tableau $T^e \in \mathsf{SShT}(\lambda/\mu,n)$, the \emph{reversal} of $T$, that is shifted Knuth equivalent to $T^*$ and dual equivalent to $T$ \cite[Theorem 2.13]{Haim92}. Since the operator $*$ preserves shifted Knuth equivalence \cite[Lemma 7.1.4]{Wor84}, the reversal operation is the coplactic extension of  evacuation, in the sense that, we may first rectify $T$, apply evacuation, and then outer \textit{jeu de taquin} slides (in the order defined by the previous rectification) to get $T^e$ with  the shape of $T$. Hence, $T^E = T^e$ for tableaux of straight shape. The reversal is a shape-preserving weight-reversing involution on shifted tableaux and it yields a bijection $T \longmapsto (T^e)^*$ between the set of LRS tableaux of shape $\lambda/\mu$ and weight $\nu$ and the set of LRS tableaux of shape $\mu^{\vee}/\lambda^{\vee}$ and weight $\nu$. Hence, we have the symmetry $f_{\mu\nu}^{\lambda} = f_{\lambda^{\vee}\nu}^{\mu^{\vee}}$.

\section{A crystal-like structure on shifted tableaux}\label{sec3}
After recalling the set up on the shifted tableau crystal $\mathcal{B}(\lambda/\mu, n)=\mathsf{SShT}(\lambda/\mu, n)$, introduced in \cite{GL19,GLP17}, we define, for each $i\in I = [n-1]$, the shifted crystal reflection operator, using the primed and unprimed crystal operators. Example \ref{exbraid} shows that these do not need to satisfy the braid relations, thus not yielding a natural action of $\mathfrak{S}_n$ on this crystal.

Let $\{e_1, \ldots,e_n\}$ be the canonical basis of $\mathbb{R}^n$. Given words $w$ and $v$ on the alphabet $[n]'$, and $i \in I$, the \emph{primed raising operator} $E_i'(w)$ is the unique word with the same standardization of $w$ and such that $wt(E_i'(w)) = wt(w) + \alpha_i$, where $\alpha_i = e_i - e_{i+1}$, if such word exists, otherwise $E'_i(w)= \emptyset$ \cite[Definition 3.3]{GLP17}.
The \emph{primed lowering operator} $F_i'(w)$ is defined  analogously using $-\alpha_i$. These notions are well defined  \cite[Lemma 3.2]{GLP17}, and as a direct consequence, $E_i'(w) = v$ if and only if $w=F_i' (v)$, for any words $w$ and $v$  \cite[Proposition 3.4]{GLP17}. The definition is extended to semistandard shifted tableaux,  $E_i' (T)$ being the shifted tableau with the same shape of $T$ and reading word $E_i' (w(T))$. These operators are well defined and are coplactic \cite[Propositions 3.6 and 3.7]{GLP17}.

The \emph{unprimed raising operators} are defined by giving conditions on the lattice walks of a word. These are extended to shifted tableaux as before. Due to lack of space, we omit these definitions and refer to \cite[Section 5.1]{GLP17} for details. The unprimed operators are coplactic and, given $T \in \mathsf{SShT}(\lambda/\mu,n)$, $E_i (T)$ and $F_i(T)\in  \mathsf{SShT}(\lambda/\mu,n)$, for all $i\in I$, whenever they are defined \cite[Theorems 5.18 and 5.35]{GLP17}.

Tableaux that differ by a sequence of unprimed raising or lowering operators are dual equivalent \cite[Corollary 5.33]{GLP17}. This is also true for the primed operators, since the standardization is unchanged. Hence, tableaux that differ by a sequence of any lowering or raising operators are dual equivalent.
 
\begin{prop}[\cite{GLP17}, Proposition 6.4]\label{highuni}
Let $\nu$ be a strict partition. The unique $T\in \mathsf{SShT}(\nu,n)$ for which $E_{i}(T) = E_i' (T)=\emptyset$, for all $i \in I$, is $Y_{\nu}$. Then, every $T\in \mathsf{SShT}(\nu,n)$ may be obtained from every other by a sequence of primed and unprimed lowering and raising operators.
\end{prop}

The set $\mathsf{SShT}(\lambda/\mu,n)$ is closed under the operators $E_i, E_i', F_i, F_i'$, for  $i \in I$.
We also have  \emph{partial length functions} \cite{GL19} given by:
\begin{align*}
\varepsilon_i'(T) &:= max\{k:\, E_i'^k(T) \neq \emptyset\} \qquad \widehat{\varepsilon_i} (T) := max\{k:\, E_i^k(T) \neq \emptyset\}\\
\varphi_i'(T) &:= max\{k:\, F_i'^k(T) \neq \emptyset\} \qquad \widehat{\varphi_i} (T) := max\{k:\, F_i^k(T) \neq \emptyset\},
\end{align*}
and \emph{total length functions} $\varepsilon_i (T)$ and $\varphi_i(T)$, defined in \cite[Section 5.1]{GLP17} via the $i$-lattice walk of $T$, for $i\in I$. The set $\mathsf{SShT}(\lambda/\mu,n)$, together with primed and unprimed operators, length functions, and weight function, is called a \emph{shifted tableau crystal} and denoted by $\mathcal{B}(\lambda/\mu,n)$. It may be regarded as a directed graph with weighted vertices, and $i$-coloured double edges, the solid ones being labelled with $i$ and the dashed ones with $i'$. For each $i\in I$, $\mathcal{B}(\lambda/\mu,n)$ may be partitioned as a set into the $\{i',i\}$-connected components underlying subsets, which are called $i$-\emph{strings}, with two possible arrangements \cite[Section 3.1]{GL19}:

\begin{center}
\includegraphics[scale=0.7]{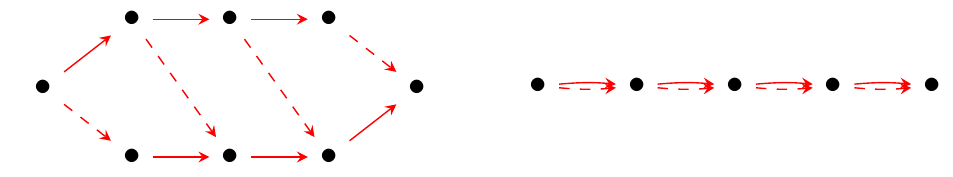}
\end{center}

The left one is called a \emph{separated $i$-string}, consisting of two $i$-labelled chains of equal length connected by $i'$-labelled edges. The smallest separated string is formed by two vertices connected by a $i'$ labelled edge. The one on the right is called a \emph{collapsed $i$-string} and is formed by a double chain both $i$- and $i'$-labelled. A single vertex is considered a collapsed string. The total length functions  can be easily formulated  in terms of $i$-doubled strings, (analogous for $\varphi_i$), $i\in I$:
$$\varepsilon_i(T) =
\begin{cases*}
\widehat{\varepsilon_i}(T)= \varepsilon_i' (T)& if $T$ is in a collapsed $i$-string\\
\widehat{\varepsilon_i}(T) + \varepsilon_i' (T) & if $T$ is in a separated $i$-string.
\end{cases*}$$

A \emph{highest weight element} (respectively \emph{lowest weight element}) of $\mathcal{B}(\lambda/\mu,n)$ is a tableau $T$ such that $E_i(T) = E_i' (T) = \emptyset$ (respectively $F_i (T) = F_i' (T) = \emptyset$), for any $i \in I$. 

\begin{prop}[\cite{GLP17}, Corollary 6.5]\label{isomhigh}
Each connected component of $\mathcal{B}(\lambda/\mu,n)$ has a unique highest weight element $T^{\mathsf{high}}$, which is a LRS tableau, and is isomorphic, as a weighted edge-labelled graph, to the shifted tableau crystal $\mathcal{B}(\nu,n)$, where $\nu = wt(T^{\mathsf{high}})$.
\end{prop}

\begin{prop}[\cite{GLP17}, Corollary 6.6]\label{compdual}
Each connected component of $\mathcal{B}(\lambda/\mu,n)$ forms a shifted dual equivalence class.
\end{prop}

\subsection{Schützenberger involution and the shifted reflection operators}
	
The Schützenberger (or Lusztig) involution is defined on the  shifted tableau crystal \cite[Section 2.3.1]{GL19} in the same fashion as for type $A$ Young tableau crystals. We realize it through shifted evacuation, for tableaux of straight shape, and through shifted reversal otherwise. Throughout this section $\nu$ will denote a strict partition.

\begin{defin}
Let $\mathcal{B}(\nu,n)$ be the shifted tableau crystal with highest weight $T^{\mathsf{high}}=Y_\nu$ and  lowest weight $T^{\mathsf{low}}=Y_\nu^E$. The \emph{Schützenberger involution} $\eta: \mathcal{B}(\nu,n) \longrightarrow \mathcal{B}(\nu,n)$ is the unique map that satisfies the following conditions, for all $T\in \mathcal{B}(\nu,n)$ and $i \in I$:

1. $E'_i \eta (T) = \eta F'_{n-i} (T)$ and $E_i \eta (T) = \eta F_{n-i} (T)$.

2. $F'_i \eta (T) = \eta E'_{n-i} (T)$ and $F_i \eta (T) = \eta E_{n-i} (T)$.

3. $wt(\eta(T)) = wt(T)^{\mathsf{rev}}$.

\end{defin}

In particular, $\eta(T^{\mathsf{high}}) = T^{\mathsf{low}}$ and we have $\varphi_i (T) = \varepsilon_{n-i} \eta (T)$ and $\varepsilon_i (T) = \varphi_{n-i} \eta (T)$. Due to Proposition \ref{isomhigh}, the involution $\eta$ may be defined in $\mathcal{B}(\lambda/\mu,n)$, by extending it to its connected components. In either cases, we denote it by $\eta$. This is indeed well defined by the next result.

\begin{prop}
The Schützenberger involution $\eta$ coincides with the evacuation $E$ in $\mathcal{B}(\nu,n)$, and with the reversal $e$ in $\mathcal{B}(\lambda/\mu,n)$.
\end{prop}

We now introduce a shifted version of the \textit{crystal reflection operators} $\sigma_i$ (\cite[Definition 2.35]{BumpSchi17}) on $\mathcal{B}(\nu,n)$, for each $i\in I$. In type $A$ Young tableau crystals, these are involutions on the crystal, so that each $i$-string is sent to itself by reflection over its middle axis, for all $i \in I$. It coincides with the restriction of the Schützenberger involution to the tableaux consisting of the letters  $i,i+1$, ignoring the remaining ones. On $\mathcal{B}(\nu,n)$, collapsed strings are similar to the $i$-strings of type $A$ crystals, hence the shifted reflection operator $\sigma_i$ is expected to resemble the one for Young tableaux. However, for separated strings, a sole reflection of the $i$-string would not coincide with the restriction of the Schützenberger involution to $\{i,i+1\}'$, hence we have the next definition.

\begin{defin}[Shifted crystal reflection operators]\label{shrefop}
Let $i\in I$ and $T\in \mathcal{B}(\nu,n)$. Let $k= \langle w(T),\alpha_i \rangle$ (usual inner product in $\mathbb{R}^n$). Define $\sigma_i(T)= T$, if $F_i(T)=F_i'(T)=\emptyset$, and otherwise according to the table below:

\begin{center}
\begin{tabular}{c|l|l}
\hspace{3em} & $F_i'(T) \neq \emptyset$ & $F_i' (T) = \emptyset$ \\ 
\hline 
\rule[-2ex]{0pt}{4.5ex}
if $k > 0$ & $F_i' F_i^{k-1} (T)$ & $E_i' F_i^{k+1} (T)$ \\ 
\rule[-2ex]{0pt}{3.5ex}
if $k = 0$ & $E_i F_i' (T)$ & $E_i' F_i (T)$ \\ 
\rule[-2ex]{0pt}{3.5ex}
if $k<0$ & $E_i^{-k+1} F_i' (T)$ & $E_i^{-k-1} E_i'(T)$ \\  
\end{tabular} 
\end{center}
\end{defin}
As the definition suggests, the shifted reflection operator $\sigma_i$ must do a double reflection, by vertical and horizontal middle axes (see Figure \ref{fig:flip}). By coplacity, the operator $\sigma_i$ is extended to $\mathcal{B}(\lambda/\nu,n)$, for $i\in I$.

\begin{figure}[h]
\centering
\includegraphics[scale=0.7]{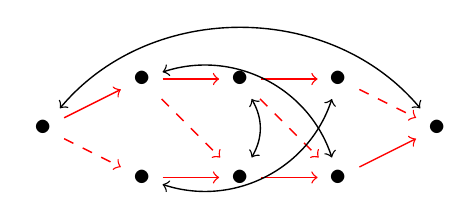}
\includegraphics[scale=0.7]{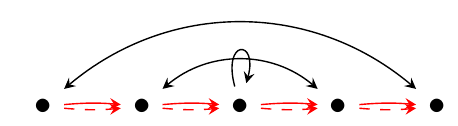}
\caption{The action of a crystal reflection operator in separated and collapsed strings, which corresponds to the Schützenberger involution.}
\label{fig:flip}

\end{figure}

\begin{prop}\label{sigmai}
For $i \in I$ and $T \in \mathcal{B}(\lambda/\mu,n)$, the operator $\sigma_i$ satisfies the following:

1. $\sigma_i$ sends each connected component of $\mathcal{B}(\lambda/\mu,n)$ to itself.

2. $\sigma_i$ takes each $i$-string to itself.

3. $\sigma_i^2 = id$ and $\sigma_i\sigma_j=\sigma_j\sigma_i$, if $|i-j| >1$.

4. $wt(\sigma_i(T)) = s_i \cdot wt(T)$, where $s_i = (i,i+1) \in \mathfrak{S}_n$.

\end{prop}

\begin{proof}
We prove the first part of the third assertion, with the case where $k>0$ and $F_i(T)\neq \emptyset$. Let $S = \sigma_i (T) = F_i' F_i^{k-1} (T)$. Then, $F_i' (S) = \emptyset$. By definition of $\sigma_i$, we prove that $wt (S) = wt(T) - k \alpha_i$. Moreover, it is easy to check that $\tilde{k} :=\langle wt(S),\alpha_i \rangle = wt(S)_i - wt(S)_{i+1} <0$. Hence, we may show that $\sigma_i(S) = \sigma_i^2 (T) =   E_i^{k-1} F_i^{k-1} (T) = T$.
\end{proof}

In what follows, $T^i$ denotes the shifted tableau obtained from $T \in \mathsf{SShT}(\lambda/\mu,n)$ considering only the boxes filled with $i'$ or $i$.

\begin{teo}\label{sigmarever}
Let $T\in \mathcal{B}(\lambda/\mu,n)$ and let $T^{i,i+1} := T^{i} \sqcup T^{i+1}$. Then,
$$\sigma_i(T) = T^{1} \sqcup \ldots \sqcup T^{i-1} \sqcup (T^{i,i+1})^e \sqcup T^{i+1} \sqcup \ldots \sqcup T^{n}.$$
\end{teo}

\begin{proof}
It suffices to prove this result for shifted tableaux on the primed alphabet of two adjacent letters. The raising and lowering operators are coplactic, so it is $\sigma_i$, thus the proof is done for rectified shifted tableaux. Furthermore, $T$ and $\sigma_i(T)$ are in the same $i$-string, hence by Proposition \ref{compdual}, $T$ and $\sigma_i(T)$ are shifted dual equivalent. It remains to show that $T^*$ and $\sigma_1(T)$ are shifted Knuth equivalent, which is done by exhibiting sequences of Knuth moves
between their words.
\end{proof}

Unlike the type $A$ crystals, the reflection operators $\sigma_i$ do not define an action of the symmetric group $\mathfrak{S}_n$ on $\mathcal{B}(\lambda/\mu,n)$ because the braid relations $\sigma_i \sigma_{i+1} \sigma_i = \sigma_{i+1} \sigma_i \sigma_{i+1}$ may not hold.

\begin{ex}\label{exbraid}
Let $\mathcal{B}(\lambda,3)$ where $\lambda=(5,3,1)$, and consider the semistandard shifted tableau
$T = \begin{ytableau}
1 & 1 & 1 & 1 & {3'}\\
\none & 2 & 2 & {3'}\\
\none & \none & 3\end{ytableau}$.
Then, we have
$
\sigma_1 \sigma_2 \sigma_1 (T)=
\begin{ytableau}
1 & 1 & 1 & 2 & 3\\
\none & 2 & {3'} & 3\\
\none & \none & 3\end{ytableau}
 \neq
 \begin{ytableau}
1 & 1 & 1 & {2'} & {3'}\\
\none & 2 & {3'} & 3\\
\none & \none & 3\end{ytableau}
= 	\sigma_2 \sigma_1 \sigma_2 (T)
 $

\end{ex}

However, we have the following result, as in \cite[Section 3.2]{AzMaCo09} for ordinary LR tableaux, ensuring that the longest permutation of $\mathfrak{S}_n$ acts on a connected component of $\mathcal{B}(\lambda/\mu,n)$ by sending the highest weight element to the lowest weight element.

\begin{teo}\label{sigmalongperm}
Let $T$ be a LRS tableau in $\mathcal{B}(\lambda/\mu,n)$. Let $\omega_0 = s_{i_1} \cdots s_{i_k}$ be the longest permutation in $\mathfrak{S}_n$. Then, $\omega_0$ acts on a connected component of $\mathcal{B}(\lambda/\mu,n)$ by sending the highest weight element $T$ to the lowest,  $\sigma_{i_1} \cdots \sigma_{i_k} (T) = T^e$.
\end{teo}

\begin{proof}
Since the operators $\sigma_i$ are coplactic, we may consider $Y_{\nu} = rect(T)$, $\nu=wt(T)$. By Proposition \ref{sigmai}, $\sigma_j$ permutes the entries $j$ and $j+1$ on the weight and keeps the shape $\nu$, and as $\omega_0$ is the longest permutation, $\sigma_{i_1} \ldots \sigma_{i_k}$ reverts the weight of $T$. The uniqueness of $Y_{\nu}$ then ensures that $\sigma_{i_1} \ldots \sigma_{i_k}Y_{\nu}=Y^E_\nu$.
\end{proof}

\section{The cactus group action on the shifted tableau crystal}\label{sec4}

We show that the restrictions of the Schützenberger involution to primed subintervals of $[n]$ define an action of the cactus group $J_n$ on $\mathcal{B}(\lambda/\mu,n)$. Halacheva \cite{Hala16} constructed this action for any $\mathfrak{g}$-crystal, for $\mathfrak{g}$ a complex reductive Lie algebra of finite dimension.

For $1 \leq p < q \leq n$, consider $[p,q]:=\{p<\cdots<q\}$. Let $\theta_{p,q \shortminus 1}$ denote the longest permutation in $\mathfrak{S}_{[p,q \shortminus 1]}$ embedded in $\mathfrak{S}_{I}$, that is, $\theta_{p, q \shortminus 1} (i)$ is $p+q-i-1$ if $i\in[p,q-1]$, and $i$ otherwise, and put $\theta := \theta_{1,n \shortminus 1}$. Given $T \in \mathcal{B}(\lambda/\mu,n)$, let $T^{p,q}:=T^p \sqcup\cdots\sqcup T^q$. In particular, $T^{1,n}=T$. By convention, we set $T^{1,0}=T^{n+1,n}  := \emptyset$ and $T^{p,p} := T^{p}$. 
To formalize the restriction of the Schützenberger involution $\eta$ to an interval $[p,q]'$, we define $\eta_{p,q}:\mathcal{B}(\nu,n)\rightarrow \mathcal{B}(\nu,n)$ as the set map such that $\eta_{p,q} (T) := T^{1,p \shortminus 1} \sqcup [T^{p,q}]^e \sqcup T^{q+1, n}$. In particular, $\eta_{p,p+1}=\sigma_p$ and $\eta_{1,n} = \eta$. This notion is extended on the connected components of $\mathcal{B}(\lambda/\mu,n)$ using Proposition \ref{isomhigh}. We denote by $\mathcal{B}_{p,q}$ the subgraph of $\mathcal{B}(\nu,n)$ obtained by removing the edges coloured in $I\setminus [p,q - 1]$, with the same vertices of $\mathcal{B}(\nu,n)$, ignoring the letters that are not in $[p, q]'$. In particular, $\mathcal{B}_{p,p+1}$ is the collection of $p$-strings. The set $\mathcal{B}(\nu,n)$ is partitioned into classes consisting of the underlying sets of the connected components $\mathcal{B}_{p,q}$ (an example with $\mathcal{B}_{2,3}$ is depicted in Figure \ref{fig:crystal}).

\begin{lema}\label{highestpq} 
Let  $1 \leq p < q \leq n$. Each connected component of $\mathcal{B}_{p,q}$ has unique highest and lowest weight elements.
\end{lema}

\begin{lema}\label{lemautil}
Let $1 \leq p < q \leq n$. Then $\eta_{p,q}: \mathcal{B}_{p,q}\rightarrow \mathcal{B}_{p,q}$ is the unique involution such that for all $T$ in each connected component of  $\mathcal{B}_{p,q}$, we have, for all $i\in [p,q-1]$:

1. $E_i' \eta_{p,q} (T) = \eta_{p,q} F_{\theta_{p,q \shortminus 1}(i)}' (T)$ and $E_i \eta_{p,q} (T) = \eta_{p,q} F_{\theta_{p,q \shortminus 1}(i)} (T)$.

2. $F_i' \eta_{p,q} (T) = \eta_{p,q} E_{\theta_{p,q \shortminus 1}(i)}' (T)$ and $F_i \eta_{p,q} (T) = \eta_{p,q} E_{\theta_{p,q \shortminus 1}(i)} (T)$.

3. $wt(\eta_{p,q} (T)) = \theta_{p,q \shortminus 1} \cdot wt (T)$.

\end{lema}

We also have that $\eta_{p,q}$ maps the highest weight of $\mathcal{B}_{p,q}$ to its lowest weight and that $\varepsilon_i (T) = \eta_{p,q} \varphi_{\theta_{p,q \shortminus 1}} (T)$ and $\varphi_i (T) = \eta_{p,q} \varepsilon_ {\theta_{p,q \shortminus 1}} (T)$, for $T \in \mathcal{B}_{p,q}$ and $i \in [p,q-1]$.

\begin{defin}[\cite{HenKam06}, Section 3.1]
The $n$-fruit \emph{cactus group} $J_n$ is the free group  with generators $s_{p,q}$,
$1 \leq p < q \leq n$, subject to the relations:
	
1. $s_{p,q}^2 = id$.

2. $s_{p,q} s_{k,l} = s_{k,l} s_{p,q}$ for $[p,q] \cap [k,l] = \emptyset$.
	
3. $s_{p,q}s_{k,l} = s_{p+q-l,p+q-k} s_{p,q}$ for $[k,l] \subseteq [p,q]$.
\end{defin}

\begin{figure}[h]
\begin{center}
\includegraphics[scale=0.5]{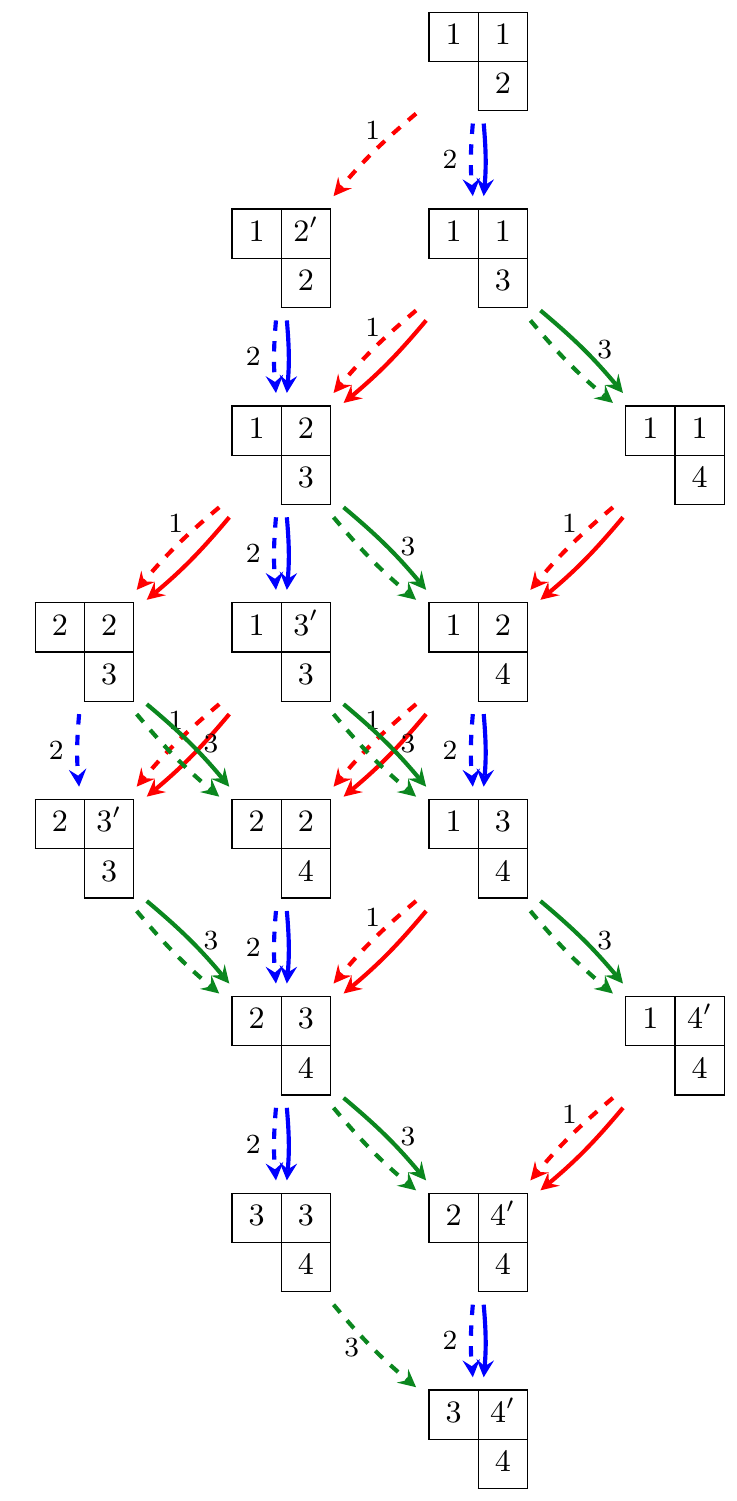}\qquad
\includegraphics[scale=0.5]{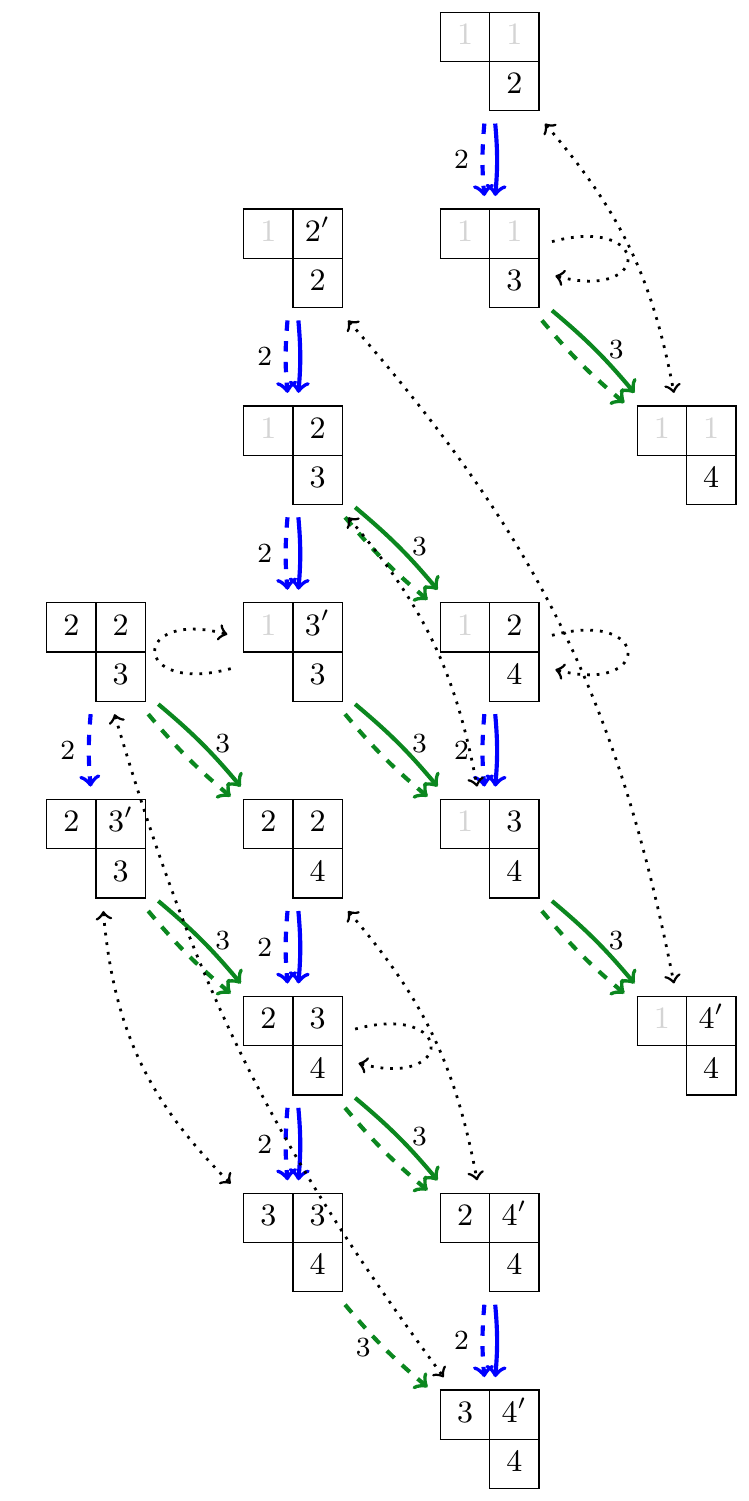}\qquad
\includegraphics[scale=0.5]{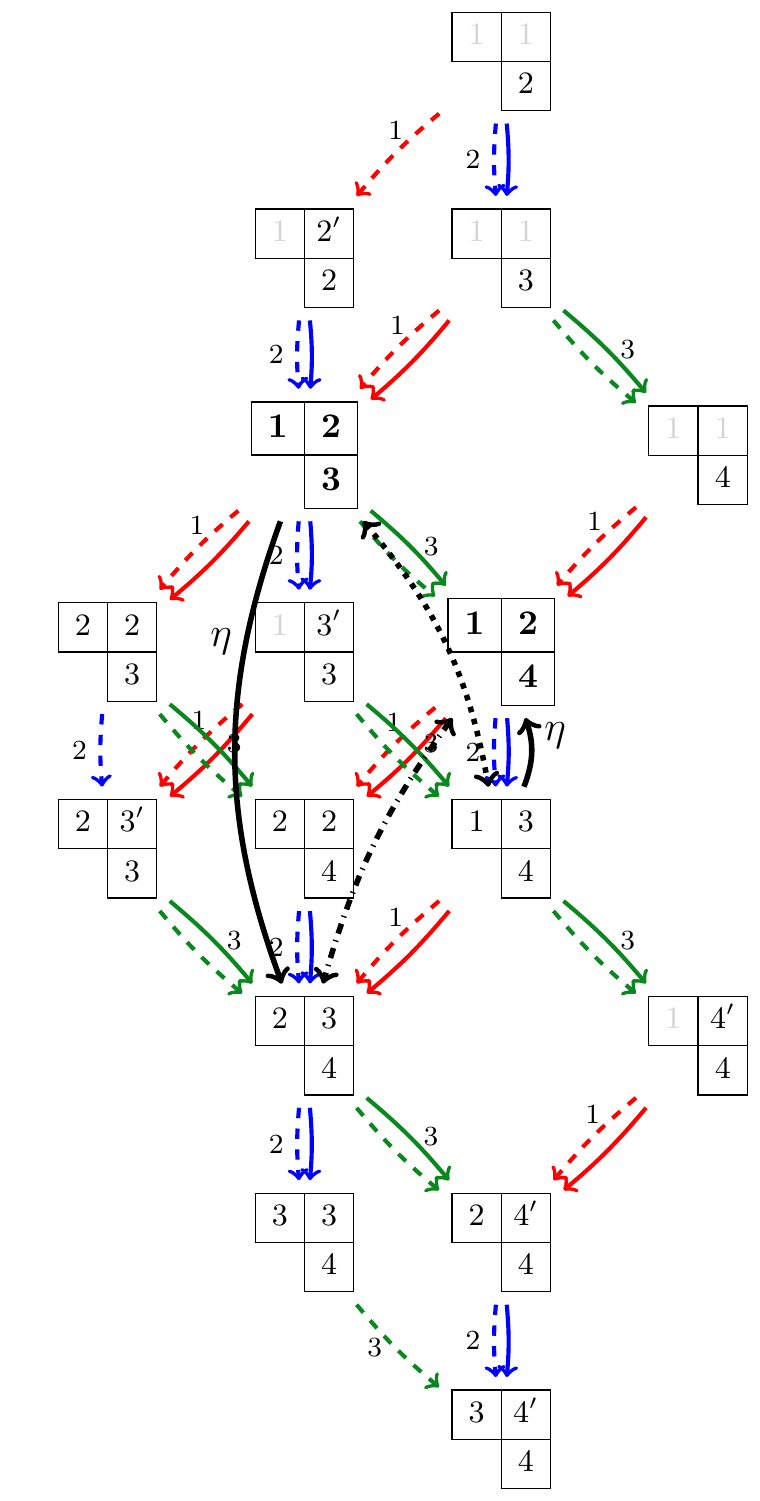}

\end{center}
\caption{On the left, the shifted tableau crystal $\mathcal{B}(\nu,4)$, for $\nu=(2,1)$, and the action of $s_{2,4}$ in the middle. On the right, an illustration of $s_{1,3} s_{1,4} = s_{1,4} s_{2,4}$.}
\label{fig:crystal}
\end{figure}

\begin{teo}[Main result]\label{cactusaction}
There exists a natural action of the $n$-fruit cactus group $J_n$ on the shifted tableau crystal $\mathcal{B}(\lambda/\mu,n)$ given by $s_{p,q} \cdot T = \eta_{p,q} (T)$, for $1 \leq p < q \leq n$.
\end{teo}

\begin{proof}
It suffices to exhibit the action in a connected component identified with $\mathcal{B}(\nu,n)$, due to Proposition \ref{isomhigh}.
Since the operator $\eta$ is an involution, it follows that $s_{p,q}^2 = id$ for all $1\leq p < q \leq n$. The second relation is a direct consequence of $\eta_{p,q}$ to act only on letters $[p, q]'$, leaving the remaining ones unchanged. For the third relation (see Figure \ref{fig:crystal} on the right), we note that the larger set is irrelevant, hence it suffices to show that $\eta \eta_{p,q} = \eta_{1+n-q,1+n-p} \eta$, for $[p,q] \subseteq [1,n]$. Let $T \in \mathcal{B}(\nu,n)$ and assume that it is in a connected component $\mathcal{B}_0$ of $\mathcal{B}_{p,q}$. By Lemma \ref{highestpq}, $\mathcal{B}_0$ has a unique highest weight $T_{0}^{\mathsf{high}}$, and lowest weight $T_{0}^{\mathsf{low}} = \eta_{p,q} (T_{0}^{\mathsf{high}})$. Moreover, $\mathcal{B}_{0} \subseteq \mathcal{B}(\nu,n)$, which has a highest weight element $Y_{\nu}$ and a lowest weight $Y_{\nu}^{\mathsf{low}} = \eta (Y_{\nu})$. Then, for some $i_1, \ldots, i_k \in [p,q-1]$, $j_1, \ldots, j_l \in [1, n-1]$, $m_i, a_j \in \{0,1\}, n_i, b_j \geq 0$, we have:
	$$T = F_{i_1}'^{m_1} F_{i_1}^{n_1} \ldots F_{i_k}'^{m_k} F_{i_k}^{n_k} (T_{0}^{\mathsf{high}})\quad \text{and} \quad
	T_{0}^{\mathsf{low}} = E_{j_1}'^{a_1} E_{j_1}^{b_1} \ldots E_{j_l}'^{a_l} E_{j_l}^{b_l} (Y_{\nu}^{\mathsf{low}})$$
Thus, using Lemma \ref{lemautil}, we may prove that:
\begin{equation}\label{RHScactus}
	\eta \eta_{p,q} (T) = F_{\theta\theta_{p,q \shortminus 1} (i_1)}'^{m_1} F_{\theta\theta_{p,q \shortminus 1}(i_1)}^{n_1} \ldots F_{\theta\theta_{p,q \shortminus 1}(i_k)}'^{m_k} F_{\theta\theta_{p,q \shortminus 1}(i_k)}^{n_k} F_{\theta(j_1)}'^{a_1} F_{\theta(j_1)}^{b_1} \ldots F_{\theta(j_l)}'^{a_l} F_{\theta(j_l)}^{b_l} (Y_{\nu})
	\end{equation}
We note that $\eta$ takes the connected component $\mathcal{B}_0$ to another connected component $\mathcal{B}_1$ of $\mathcal{B}_{n-q+1,n-p+1}$. We have that $\eta$ interchanges the highest and lowest weight elements in $\mathcal{B}_0$ and $\mathcal{B}_1$, thus, $\eta(T_0^{\mathsf{low}})$ and $\eta (T_0^{\mathsf{high}})$ are, respectively, the highest and lowest weight elements of $\mathcal{B}_1$. Since $\mathcal{B}_1$ is a component of $\mathcal{B}_{n-q+1,n-p+1}$, then $\eta_{n-q+1,n-p+1}$ maps its lowest weight to its highest weight, hence $\eta_{n-q+1,n-p+1} \eta (T_0^{\mathsf{high}})$ is the highest weight in $\mathcal{B}_1$. Then, we have $\eta_{n-q+1,n-p+1} \eta (T_{0}^{\mathsf{high}}) = \eta (T_{0}^{\mathsf{low}})$ and we may prove that:
	\begin{equation}\label{LHScactus}
	\eta_{n-q+1,n-p+1} \eta (T) = F_{\theta\theta_{p,q \shortminus 1}(i_1)}'^{m_1} F_{\theta\theta_{p,q \shortminus 1}(i_1)}^{n_1} \ldots F_{\theta\theta_{p,q \shortminus 1}(i_k)}'^{m_k} F_{\theta\theta_{p,q \shortminus 1}(i_k)}^{n_k} F_{\theta(j_1)}'^{a_1} F_{\theta(j_1)}^{b_1} \ldots F_{\theta(j_l)}'^{a_l} F_{\theta(j_l)}^{b_l} (Y_{\nu})
	\end{equation}
Hence, by \eqref{RHScactus} and \eqref{LHScactus}, we have $\eta\eta_{p,q}(T) = \eta_{n-q+1,n-p+1} \eta (T)$.
\end{proof}

\acknowledgements{The author wishes to express her gratitude to her supervisors Olga Azenhas and Maria Manuel Torres and to acknowledge the hospitality of the Department of Mathematics of University of Coimbra.}

%\printbibliography
{\small \bibliographystyle{plain}\bibliography{bibliography}}

\end{document}